\documentclass{amsart}

\usepackage{amsmath, amsthm, amssymb}
\usepackage{mathtools}
\usepackage{comment}
 \usepackage{todonotes}
\usepackage{cite}
\usepackage{tikz}
\usetikzlibrary{positioning}

\theoremstyle{definition}

\newtheorem{thm}{Theorem}[section]

\newtheorem*{thm*}{Theorem}
\newtheorem{lemma}[thm]{Lemma}

\newtheorem{defn}[thm]{Definition}
\newtheorem{claim}[thm]{Claim}
\newtheorem{prop}[thm]{Proposition}
\newtheorem{cor}[thm]{Corollary}

\newtheorem{remark}[thm]{Remark}
\newtheorem{fact}[thm]{Fact}

\newtheorem{example}[thm]{Example}

\newtheorem{obs}[thm]{Observation}

\newcommand{\thistheoremname}{}
\newtheorem{genericthm}[thm]{\thistheoremname}

\renewcommand{\subset}{\subseteq}
\renewcommand{\P}{\mathbb{P}}

\newcommand{\rest}{\restriction}

\newcommand\force{\Vdash}

\newcommand\R{\mathbb{R}}
\newcommand\N{\mathbb{N}}
\newcommand\Z{\mathbb{Z}}
\newcommand\Q{\mathbb{Q}}
\newcommand\C{\mathbb{C}}

\DeclareMathOperator{\dom}{dom}

\DeclareMathOperator{\hod}{HOD}

\newcommand{\set}[2]{ \left\{ #1 :\, #2 \right\} }
\newcommand{\seqq}[2]{ \left( #1 :\, #2\right) }
\newcommand{\seq}[1]{ \left( #1 \right) }

\title{Ergodicity as an obstruction to classification by higher-rank invariants}

\author{Assaf Shani}
\address{Department of Mathematics and Statistics, Concordia University, Montreal, QC  H3G 1M8, Canada}
\email{assaf.shani@concordia.ca}
\urladdr{https://sites.google.com/view/assaf-shani/}

\date{\today}



\thanks{Research partially supported by NSERC grant RGPIN-2024-05827
.}

\begin{document}

\begin{abstract}
Given a collection of mathematical objects and a natural notion of equivalence (such as isomorphism), we want to find a complete classification.
This paper presents methods for proving anti-classification results, that is, when a complete classification is not possible using certain classifying invariants.

The problem of using real numbers as classifying invariants is well understood. In this case, the obstruction to classifiability is ergodicity. We consider classifiability by higher rank objects, such as countable sets of real numbers, and present a simple ergodicity-like obstruction for classifiability. This method is particularly useful when applied in various models of ZF where even weak fragments of the axiom of choice fail.
\end{abstract}
\maketitle

\section{Introduction}~\label{section: Introduction}

Let $E$ be an equivalence relation on $X$. A \textbf{complete classification} of $E$ is a map $c\colon X\to I$ satisfying
\begin{equation*}
    x\mathrel{E}y\iff c(x)=c(y).
\end{equation*}
The members of $I$ are said to be \textbf{classifying invariants} for $E$.
Given an equivalence relation $E$, one would like to classify it using the simplest possible invariants. 
For this to make sense, some restrictions must be placed on the classifying map $c$. The study of equivalence relations on Polish spaces provides a precise framework to study classification problem and calibrate their complexity.

Let $E$ be an equivalence relation on a Polish space $X$. $E$ is said to be \textbf{concretely classifiable} if there is a Polish space $Y$ and a Borel map $c\colon X\to Y$ which is a complete classification of $E$. Equivalently, we may assume that the space $Y$ is the space of reals $\mathbb{R}$, that is, $E$ admits real numbers as classifying invariants (see \cite[Section 5.4]{Gao09}).

Concrete classifiability is well understood. The canonical obstruction for concrete classification is \textbf{ergodicity}. Suppose $a\colon\Gamma\curvearrowright X$ is a continuous action of a countable group $\Gamma$ on a Polish space $X$. The action is \textbf{generically ergodic} if for any non-empty open sets $U,V\subset X$ there is $\gamma\in \Gamma$ so that $\gamma\cdot U \cap V \not =\emptyset$. Equivalently: if every invariant Borel subset of $X$ is either meager or comeager. Similarly, given a probability measure $\mu$ on $X$, the action is \textbf{$\mu$-ergodic} if any invariant Borel set is either measure 0 or measure 1.
Let $E_a$ be the induced orbit equivalence relation on $X$: $x\mathrel{E_a}y\iff\exists\gamma\in\Gamma (\gamma\cdot x = y)$.

\begin{fact}\label{fact: ergodic not concrete class}
    Suppose $a\colon \Gamma\curvearrowright X$ is generically ergodic, or $\mu$-ergodic. Then the induced orbit equivalence relation $E_a$ is not concretely classifiable.
\end{fact}

A canonical example of ergodicity is the equivalence relation $E_0$, defined on $2^\N$ as eventual equality between binary sequences. $E_0$ can be seen as the orbit equivalence relation induced by an action of $\bigoplus_{n\in\N} \Z_2$ on $2^\N$. This action is generically ergodic, and ergodic with respect to the coin-flipping measure on $2^\N$. 

\subsection{Borel reducibility}

Let $E$ and $F$ be equivalence relations on Polish spaces $X$ and $Y$ respectively. A map $f\colon X\to Y$ is said to be a \textbf{reduction} of $E$ to $F$ if for any $x_1,x_2\in X$, 
\begin{equation*}
    x_1\mathrel{E}x_2\iff f(x_1)\mathrel{F}f(x_2).
\end{equation*}
We say that $E$ is \textbf{Borel reducible} to $F$, denoted ${E}\leq_B{F}$ if there is a Borel measurable function which is a reduction of $E$ to $F$. Borel reducibility is the most central concept in the study of equivalence relations on Polish spaces.

Let $=_\R$ be the equality relation on $\R$. Note that $E$ is concretely classifiable if and only if ${E}\leq_B{=_\R}$.

\subsection{Classification by hereditarily countable sets}

An equivalence relation is \textbf{classifiable by countable structures} if it is Borel reducible to an isomorphism relation on a space of all countable $\mathcal{L}$-structures, for some countable language $\mathcal{L}$. (See \cite[12.3]{Kano08}, \cite[3.6]{Gao09}, \cite{Hjo00}). In this case, the Scott analysis provides a complete classification. The classifying invariants are Scott sentences, which can be coded as hereditarily countable sets.

The \textbf{Friedman-Stanley jump hierarchy} stratifies the above notion according to the rank of the hereditarily countable set.

Given an equivalence relation $E$ on $X$, the \textbf{Friedman-Stanley jump} of $E$ is the equivalence relation $E^+$ on $X^\N$ defined by
\begin{equation*}
    x\mathrel{E^+}y \iff \forall n \exists m (x(n)\mathrel{E} y(n)) \textrm{ and } \forall n \exists m (y(n) \mathrel{E} x(m)),
\end{equation*}
equivalently, if  $\set{[x(n)]_E}{n\in\N}=\set{[y(n)]_E}{n\in\N}$.
Friedman and Stanley~\cite{FS89} proved that this is a \emph{jump operator} on Borel equivalence relations. That is, $E <_B E^+$ for any Borel equivalence relation $E$.

\begin{example}\label{example: =+ unpinned}
    Let $=^+$ be $(=_\R)^+$, defined on $\R^\N$ by $x =^+ y\iff \set{x(n)}{n\in\N} = \set{y(n)}{n\in\N}$. The map $x\mapsto \set{x(n)}{n\in\N}$ is a natural complete classification of $=^+$ using countable sets of reals as classifying invariants.
\end{example}
A classification problem is considered \emph{classifiable using countable sets of reals as complete invariants} if it is Borel reducible to $=^{+}$; ``classifiable using countable sets of countable sets of reals as complete invariants'' if it is Borel reducible to $(=^+)^+$; and so on.
The iterated Friedman-Stanley jumps, $=^{+\alpha}$, are defined recursively, where $=^{+1}$ is $=^+$, $=^{+(\alpha+1)}$ is $(=^{+\alpha})^+$, and we take products at limit stages (see \cite[12.2.6]{Gao09}).

\begin{fact}[{Friedman~\cite[Theorem~1.5]{Friedman2000}}]
    A Borel equivalence relation which is classifiable by countable structures is Borel reducible to $=^{+\alpha}$ for some countable ordinal $\alpha$.
\end{fact}

\subsection{The main result}\label{Section: main result}

The goal of this paper is to provide a generalized ergodicity criterion which
\begin{itemize}
    \item implies unclassifiability by more complex invariants, such as sets of real numbers;
    \item is simple to apply, like the above ergodicity argument.
\end{itemize}

We state the main anti-classification theorem here, though the key concept of an \emph{ergodic name for a virtual class} is introduced later in Section~\ref{Section: generalized ergodicity}. (See Remark~\ref{remark: key idea for generalization and virtual classes} for more intuition.)

\begin{thm}[ZF]\label{thm: anti classification subsets of V}
 Let $\P$ be a forcing poset and $\tau$ a $\P$-name for a  virtual $E$-class. Assume that $\tau$ is unstable and $(\P,E)$-ergodic. Then there is no absolute complete classification $c\colon X \to I$ of $E$ such that in any $\P$-generic extension of $V$, the invariants in $I$ are subsets of $V$.
\end{thm}
The theorem will be applied in one of the following two forms. Here $\mathcal{P}^\alpha$ denotes the iterated powerset operation, and $\mathcal{P}(\N)$ is identified with $\R$. Recall that $=^+$ admits classifying invariants which are subsets of $\mathcal{P}(\N)$. Similarly, $=^{+\alpha}$ admits classifying invariants which are subsets of $\mathcal{P}^{\alpha}(\N)$. 


\begin{thm}\label{thm: no classification P(N)}
     Let $\P$ be a forcing poset and $\tau$ a $\P$-name for a  virtual $E$-class. Assume that $\tau$ is unstable and $(\P,E)$-ergodic. Assume further that forcing with $\P$ adds no new members of $\mathcal{P}^{\alpha}(\N)$.    
    Then there is no absolute complete classification of $E$ using subsets of $\mathcal{P}^\alpha(\N)$ as invariants. In particular, ${E}\not\leq_B {=^{+\alpha}}$.
\end{thm}

The typical application will be to conclude irreducibility to some $=^{+\alpha}$. In this paper we will focus on anti-classification by sets of reals, $\alpha=1$, in which case the conclusion is that ${E}\not\leq_B {=^+}$. Often times we can get a bit more out of Theorem~\ref{thm: anti classification subsets of V}:

\begin{thm}[ZF]\label{thm: no classification P(On)}
     Let $\P$ be a forcing poset and $\tau$ a $\P$-name for a  virtual $E$-class. Assume that $\tau$ is unstable and $(\P,E)$-ergodic. Assume further that forcing with $\P$ adds no new members of $\mathcal{P}^{\alpha}(\mathrm{On})$.    
    Then there is no absolute complete classification of $E$ using subsets of $\mathcal{P}^\alpha(\mathrm{On})$ as invariants. In particular, ${E}\not\leq_B {=^{+\alpha}}$.
\end{thm}

We emphasize that the theorems hold in ZF since it is often most fruitful to apply these when working over a choiceless model, as shown in the examples below. For Theorem~\ref{thm: no classification P(On)} specifically, choice must fail in the ground model for the hypothesis to hold, e.g. adding no sets of ordinals via forcing.
The irreducibility or unclassifiability results which we care about are absolute, and therefore their truth does not depend on the model in which we apply the theorem.


\begin{remark}
    The main theorems above are an evolution of results from~\cite{Shani-reducibility-symmetric-models-21,Shani-countable-products, Anti-classification-CMMS,Allison-Shani-abelian-products}. In the examples below we reproduce a few results from these papers. 
    The scope of the main result here is less general than~\cite[Theorem~1.9]{Shani-reducibility-symmetric-models-21}, in the sense that it deals only with reducibility to $=^{+\alpha}$, yet more general than the techniques in~\cite{Allison-Shani-abelian-products}, as there is no restriction on $E$. The key benefits of the anti-classification criterion presented here is its simplicity, similarity to the usual ergodicity argument, and applicability. We believe that this presentation will be useful for future applications.
\end{remark}

\section{Preliminaries}

\subsection{Abstract classification}
In general, for a complete classification to be `reasonable', we would like to avoid the following two classifications:
\begin{enumerate}
    \item Appealing to the axiom of choice, there is a map choosing one element out of each equivalence class. This map is a complete classification, but does not provide any insights regarding the equivalence relation $E$.
    \item The map sending $x$ to its equivalence class $[x]_E=\set{y\in X}{x\mathrel{E}y}$, is a complete classification, which again does not provide meaningful insights.
\end{enumerate}

In this paper we will work with \textbf{absolute complete classifications}. This is meant to be a weak condition for a complete classification to be `reasonable', which always holds in specific examples (in particular, for classification by countable structures). See~\cite{Shani2024-ClassifyingInvariantsE1TailGenericReal} for a detailed discussion.

    \begin{defn}\label{defn: absolute classification}
Let $E$ be an equivalence relation on a Polish space $X$. Say that $c\colon X \to I$ is an \textbf{absolute complete classification} if
\begin{enumerate}
    \item The map $c\colon X\to I$ is defined in some set theoretic way, $c(x)=A\iff \psi(x,A)$, for some set theoretic formula $\psi$ (using some parameters).
    \item In any ZF extension $V\subset N$, $\psi$ still defines a map which is a complete classification of $E$.
    \item Given ZF models $V\subset N\subset W$, if $x$ and $A$ are in $N$ and $N\models \psi(x,A)$, then $W\models \psi(x,A)$.
\end{enumerate}
Similarly, a map $c\colon X\to I$ is an \textbf{absolute $E$-invariant map} if the it satisfies the above with `complete classification' replaced by `$E$-invariant', that is, $x_1\mathrel{E}x_2 \implies c(x_1) = c(x_2)$.
\end{defn}
\begin{remark}
    In the above definitions, the spaces $X$, $I$ and the map $c$ are interpreted according to their definition in the various models. For the members of these spaces, such as $x\in X$ and $A\in C$ in clause~(3), we consider these as fixed objects.
\end{remark}

\begin{example}
    For an isomorphism relation $\cong_{\mathcal{L}}$, the map
\begin{equation*}
    M\mapsto \varphi_M,
\end{equation*}
sending a model to its Scott sentence, is an absolute complete classification as above.
\end{example}
\begin{example}[Absolute classification for Friedman-Stanley jumps]\label{example: abs class for FS jumps}
 \,
\begin{itemize}
    \item The map $\R^\N \to \mathcal{P}_{\aleph_0}(\R)$, $x\mapsto \set{x(n)}{n\in\N}$, is an absolute complete classification of $=^+$.
    \item The map $(\R^\N)^\N \to \mathcal{P}_{\aleph_0}\mathcal{P}_{\aleph_0}(\R)$,
\begin{equation*}
     x\mapsto \set{\set{x(n)(m)}{m\in\omega}}{n\in\omega}
\end{equation*} is an absolute complete classification of $=^{++}$
\end{itemize}
\end{example}

\begin{remark}
    Borel reducibility ${E}\leq_B{=^+}$ gives an absolute complete classification of $E$ using $\mathcal{P}_{\aleph_0}(\mathbb{R})$, countable sets of reals, as classifying invariants. Recall that a map $f\colon X\to Y$ is a \textbf{homomorphism} from $E$ to $F$ if for any $x_1,x_2\in X$,
\begin{equation*}
x_1\mathrel{E}x_2 \implies f(x_1)\mathrel{F} f(x_2).    
\end{equation*}
We write $f\colon E\to_B F$ to denote that $f$ is a Borel homomorphism from $E$ to $F$. A Borel homomorphism $f\colon {E}\to_B {=^+}$ gives an absolute $E$-invariant map $X\to \mathcal{P}_{\aleph_0}(\mathbb{R})$.

\end{remark}

\begin{example}[Absolute classification for CBERs]\label{example: abs class for CBERs}
    Suppose $E$ is a countable Borel equivalence relation (CBER), that is, every $E$-class is countable. In this case, the map $x\mapsto [x]_E$ \emph{is} an absolute complete classification. The point is that the $E$-class $[x]_E$ does not change in generic extensions, and therefore condition (3) holds. 
\end{example}

In our main Theorems, Section~\ref{Section: main result}, we provide criteria to conclude that a certain absolute complete classification is impossible. The main corollary is the irreducibility to some Friedman-Stanley jump. In this context, the use of absolute complete classifications is mostly for clarity of the proofs.

Classification using countable ordinals (invariants in $\omega_1$), countable subsets of countable ordinals (invariants in ${\mathcal{P}_{\aleph_0}(\omega_1)}$), and so on, are also of interest. Proving non-classifiability using $\mathrm{On}, \mathcal{P}(\mathrm{On}),$ and so on (as in Theorem~\ref{thm: no classification P(On)}) excludes such classifications.

For example, the equivalence relation $E_{\mathrm{Ulm}}$ is an analytic relation that codes the “equality relation between countable sets of countable ordinals” (sometimes denoted $\mathrm{Id}(2^{<\omega_1})$).
In the context of non-Borel analytic equivalence relations, $E_{\mathrm{Ulm}}$ plays a role analogous to that of $=_\R$ among Borel equivalence relations. The equivalence relation $E_{\mathrm{Ulm}}^+$ then plays a role similar to that of $=^+$. In the context of non-Borel equivalence relations the common notion of reducibility is absolute ${\Delta}^1_2$-reducibility. Indeed `no absolute classification using invariants in $\mathcal{P}^2(\mathrm{On})$' implies no absolute $\Delta^1_2$-reduction to $E_{\mathrm{Ulm}}^{+}$.
For example, we show in Corollary~\ref{cor: =++ unclass sets of sets ord} below that $=^{++}$ is not classifiable using invariants in $\mathcal{P}^2(\mathrm{On})$.
This is a strengthening of the irreducibility to $=^+$.

\subsection{Ergodicity: A forcing point of view}\label{section: ergodicity forcing pov}
In this section we present the unclassifiability of $E_0$ through an equivalent forcing point of view. In this presentation, the generalization given in Section~\ref{Section: generalized ergodicity} is quite natural.

Fact~\ref{fact: ergodic not concrete class} is usually presented as a consequence of the following.
\begin{fact}\label{fact: basic ergodic fact}
Let $\Gamma$ be a countable group and $X$ a Polish space. Assume that $a\colon \Gamma\curvearrowright X$ a generically ergodic action. Then any $E_a$-invariant Borel map $X \to \R$ is constant on a comeager set.  

Similarly, in the measure theoretic context, if $\Gamma$ is a countable group, $a\colon \Gamma\curvearrowright X$ is a $\mu$-ergodic action on a standard measure space $(X,\mu)$, then any $E_a$-invariant Borel map $X \to \R$ is constant on a full measure set.

In either case, as each orbit is countable, we conclude that there is no Borel reduction from $E_a$ to $=_\R$.
\end{fact}

Recall that a real number may be identified as a subset of $\N$. The unclassifiability of $E_0$ is true in greater generality. For example, in $L(\R)$ there is no injective map from $2^\N / E_0$ into $\mathcal{P}(\alpha)$ for any ordinal $\alpha$. (See~\cite{Hjorth-1995-dichotomy-definable-universe}, \cite{Caicedo-Ketchersid-2010}. Also~\cite[7.7]{Shani2024-ClassifyingInvariantsE1TailGenericReal}.)

\begin{prop}\label{prop: no classification E0 P(On)}
    Any absolute invariant map from $E_0$ to $\mathcal{P}(\mathrm{On})$ must send a comeager subset of $2^\N$ to the same set of ordinals. 
    In particular, there is no absolute complete classification of $E_0$ using invariants in $\mathcal{P}(\mathrm{On})$. 
\end{prop}
The proposition relies on the \emph{weak homogeneity} of Cohen forcing $\P$ (adding a member of $2^\N$ by finite approximations). Recall that a poset $\P$ is \textbf{weakly homogeneous} if for any two conditions $p,q\in\P$ there is an automorphism $\gamma$ of $\P$ so that $\gamma\cdot p$ is compatible with $q$. 
When forcing with a weakly homogeneous poset, the generic filter is highly undefinable. This is captured by the following well known fact, going back to early work on forcing by Feferman and Levy.\footnote{Let $x\in 2^\N$ be a Cohen generic real over $L$. Feferman showed that there is no definable well ordering of the reals in $L[x]$.
Levy showed that in fact any hereditarily ordinal definable set in $L[x]$ is in $L$, the proof of which is essentially Fact~\ref{fact: weak homogeneity}. See \cite[Section~7]{Kan06}.} 

\begin{fact}\label{fact: weak homogeneity}
    Assume that $\P$ is weakly homogeneous. Fix $G\subset \P$ a generic filter over $V$. If $Z\in V[G]$ is a subset of $V$ which is definable (in $V[G]$) using parameters in $V$, then $Z\in V$.
\end{fact}

The following is a mild generalization of Fact~\ref{fact: weak homogeneity}, allowing the the parameter $[x]_{E_0}$ in definitions, as the relevant automorphisms preserve $[x]_{E_0}$. 

\begin{claim}\label{claim: levy hod in L[x]}
Let $x\in 2^\N$ be a Cohen generic real over $V$. Let $\sigma$ be a name so that it is forced that $\sigma$ is a subset of the ground model and that $\sigma$ is definable using $[x]_{E_0}$ and parameters from $V$. Then there is a set $Z\in V$ so that $\sigma = \check{Z}$ is forced.
\end{claim}

    

Cohen forcing $\P$ is equivalent to the topological version of Cohen forcing for the space $2^\N$. We will use the following fact: Let $\dot{x}$ be the name for the generic Cohen real added by $\P$. If $\P\force \varphi(\dot{x})$, for some formula $\varphi$ using ground model parameters, then $\varphi(x)$ is true, in the ground model, for a comeager set of reals.

\begin{proof}[Proof of Proposition~\ref{prop: no classification E0 P(On)}]
    Assume that $x\mapsto B_x$ is an absolute $E_0$-invariant map where $B_x$ is always a set of ordinals. 
    Let $x\in 2^\N$ be a Cohen generic real over $V$, and let $B = B_x$.
    Note that, in $V[x]$, $B$ is definable from $[x]_{E_0}$ and parameters in $V$. Specifically, $B$ is defined as the unique set so that $B = B_y$ for any $y\in [x]_{E_0}$.
    
    Since $B$ is a set of ordinals, we conclude from Claim~\ref{claim: levy hod in L[x]} that $B\in V$ and $B_{\dot{x}} = \check{B}$ is forced, and therefore there is a comeager set of $x\in 2^\N$ for which $B_x = B$. 
\end{proof}
We will repeat these ideas in much greater generality below, which will allow us to prove the main anti-classifiability results from the introduction.

\begin{remark}\label{remark: key idea for generalization and virtual classes}
The key idea is quite simple: the ergodicity (weak homogeneity) prevented a classification using classifying invariants which are subset of the ground model. Since forcing does not change the ordinals, we have that a set of ordinals cannot be a classifying invariant. (This is captured by Theorem~\ref{thm: anti classification subsets of V}.)

The way forward seems obvious now: To prevent classifying invariant which are sets of reals (sets of sets of ordinals), we just need any set of reals to be a subset of the ground model. This is guaranteed if the poset $\P$ does not add any reals! (This is captured by Theorem~\ref{thm: no classification P(N)}.) However, the whole point seems to be that the poset adds a new member of the underlying Polish space.

This point leads us to consider \emph{virtual equivalence classes}. While it is not possible to add a new equivalence class without adding reals, it will be quite natural to add virtual equivalence classes without adding reals. Moreover, the discussion above about weak homogeneity and anti-classification generalizes naturally. This is done in Section~\ref{Section: generalized ergodicity}.
\end{remark}
First we recall the definition of virtual equivalence classes.

\subsection{Virtual equivalence classes}

\begin{defn}[{See~\cite[17.1.2]{Kano08},\cite[2.1.1]{Larson-Zapletal-Geometric-Set-Theory}}]\label{defn: virtual class}
Let $\P$ be a forcing poset and $\tau$ and $\P$-name so that $\P\force \tau \in X$.
Say that $(\P,\tau)$ is a \textbf{stable virtual $E$-class} if 
\begin{equation*}
    \P\times \P \force \tau_l\mathrel{E}\tau_r.
\end{equation*}
Here $\tau_l, \tau_r$ are $\P\times\P$-names interpreting $\tau$ according to the left and right generic filters respectively.
\end{defn}
For applications of virtual classes, see also~\cite{Hjorth-1998-Absoluteness-principle}, \cite{Shani2026-Virtual-classes}.
\begin{example}
    For $x\in X$, and any poset $\P$, the pair $(\P,\check{x})$ is a stable virtual $E$-class.
\end{example}
So each member of $X$ represents a stable virtual $E$-class. Therefore the space of all stable virtual classes, `the virtual realm', can be seen as an extension of $X$. 
The equivalence relation $E$ can be extended to the virtual realm. This was first does by Larson and Zapletal.

\begin{defn}[Larson-Zapletal~{\cite[2.1.4]{Larson-Zapletal-Geometric-Set-Theory}}] Let $(\P,\tau)$ and $(\Q,\sigma)$ be stable virtual $E$-classes.
Say that $(\P,\tau)$ and $(\Q,\sigma)$ are \emph{$E$-equivalent}\footnote{In~\cite{Larson-Zapletal-Geometric-Set-Theory} the extension of $E$ to virtual classes is denoted by $\bar{E}$. Furthermore, there the pairs $(\P,\tau)$ are called \emph{pins}, and a \emph{virtual $E$-class} refers to the $\bar{E}$-equivalence class of $(\P,\tau)$.}, denoted $(\P,\tau)\mathrel{E}(\Q,\sigma)$ if
\begin{equation*}
    \P\times\Q\force \tau\mathrel{E}\sigma,
\end{equation*}
where we identify $\tau,\sigma$ with the corresponding $\P\times\Q$-names which interpret the left and right generics respectively.
\end{defn}

Furthermore, any absolute $E$-invariant map can be extended to the virtual realm. Assume that $c\colon X\to I$ is a an absolute $E$-invariant map and $(\P,\tau)$ is a stable virtual $E$-class. In this case, $\P\times \P \force c(\tau_l) = c(\tau_r)$. A mutual genericity argument shows that there must be some $B\in I$ so that $\P \force c(\tau) = \check{B}$.

\begin{defn}\label{defn: extending absolute map to virt class}
    Assume that $c\colon X\to I$ is a an absolute $E$-invariant map. We define $c$ on stable virtual $E$-classes by $c((\P,\tau)) = B$ for the set $B\in I$ so that $\P\force c(\tau) = \check{B}$.
\end{defn}
Note that if $c$ is an absolute complete classification of $E$ then the map $c$ is a complete classification of the extended equivalence relation $E$ on stable virtual $E$-classes.

\section{Generalized ergodicity}\label{Section: generalized ergodicity}
In Section~\ref{section: ergodicity forcing pov} we dealt with a name $\dot{X}$ for a \emph{member} of the space $X$. We now consider names for a \emph{virtual class}.

\begin{defn}
    Let $\P$ be a forcing poset, $\tau$ a $\P$-name for a stable virtual $E$-class.
    
    \begin{enumerate}
        \item Given an automorphism $\gamma$ of $\P$, say that $\gamma$ \emph{respects $(E,\tau)$} if $\tau[G] \mathrel{E} \tau[\gamma\cdot G]$ for any $\P$-generic $G$.  
        \item Say that $\tau$ is {\bf $(\P,E)$-ergodic} if for any $p,q\in\P$ there is an automorphism $\gamma$ of $\P$ so that $\gamma\cdot p$ is compatible with $q$ and $\gamma$ respects $(E,\tau)$.
        \item Say that $\tau$ is unstable if $\tau[G]\not\mathrel{E}\tau[H]$, whenever $G,H$ are mutually $\P$-generics.
    \end{enumerate}
\end{defn}
Recall that, for a Polish space $X$, \textbf{Cohen forcing on $X$} is the poset $\P$ of non-meager Borel sets, where $p$ extends $q$ if $p\setminus q$ is meager. Given a measure $\mu$ on $X$, \textbf{Random real forcing on $X$} is the poset $\P$ of all positive measure Borel sets, where $p$ extends $q$ if $p\setminus q$ is null. See~\cite[Section~26]{Jec03},~\cite{Zapletal-idealized-2008}. In both cases there is a natural name $\dot{x}$ for a generic member of $X$.
\begin{example}
Recall that each member $x$ of $X$ represents a stable virtual class. In particular, if $\dot{x}$ is a name for a member of $X$, then it may be considered as a name for a stable virtual class. This is the case in the following examples. 
\begin{enumerate}
    \item Suppose is $E$ is the equality equivalence relation on $X$ and $\tau$ is $\check{x}$ for some $x\in X$ in the ground model. Then $\tau$ is $(\P,E)$-ergodic if and only if $\P$ is weakly homogeneous.
    \item  Let $X$ be a Polish space and $a\colon\Gamma\curvearrowright X$ a generically ergodic action by homeomorphisms. Let $\P$ be Cohen forcing on $X$ and $\dot{x}$ the name for the generic Cohen real. Then $\dot{x}$ is unstable and $(\P,E_a)$-ergodic.
    \item Let $(X,\nu)$ be a standard measure space and $a\colon\Gamma\curvearrowright X$ an ergodic group action by measure preserving transformations. Let $\P$ be Random real forcing on $X$, $\dot{x}$ the name for the generic random real. Then $\dot{x}$ is unstable and $(\P,E_a)$-ergodic.
\end{enumerate}
\end{example}


In particular, for the poset $\P$ of all finite partial functions from $\N$ to $2$, which is equivalent to Cohen forcing on $2^\N$, the name $\dot{x}$ for the generic sequence in $2^\N$ is $(\P,E_0)$-ergodic. 
The following is a generalization of Proposition~\ref{prop: no classification E0 P(On)}.

\begin{lemma}[ZF]\label{thm: general P ergodicity}
    Let $\P$ be a forcing poset and $\tau$ a $(\P,E)$-ergodic $\P$-name for a  virtual $E$-class. If $c\colon X\to I$ is an absolute $E$-invariant map, where in any $\P$-generic extension $c(\tau)$ is a subset of the ground model, then there is a fixed $A\in I$ in the ground model so that $c(\tau) = \check{A}$ is forced by $\P$.
\end{lemma}
As an immediate corollary we get the anti-classification criterion, Theorem~\ref{thm: anti classification subsets of V}.
\begin{proof}[Proof of Theorem~\ref{thm: anti classification subsets of V}]
    Let $c\colon X\to I$ be an absolute $E$-invariant map so that in any $\P$-generic extension of $V$, the invariants in $I$ are subsets of $V$. We show that $c$ is not a complete classification.
    
    By the lemma, there is a set $A$ in $V$ so that $\P$ forces that $c(\tau) = \check{A}$. Let $G, H$ be mutually $\P$-generics over $V$. Since $\tau$ is unstable, $\tau[G]\not\mathrel{E}\tau[H]$. However, $c(\tau[G]) = A = c(\tau[H])$, so $c$ is not an absolute complete classification. 
\end{proof}
\begin{proof}[Proof of Lemma~\ref{thm: general P ergodicity}]
     \begin{claim}
        For any $a$ in $V$ there are no two conditions $p,q\in\P$ so that 
        \begin{equation*}
            p\force \check{a}\in c(\tau), \quad q\force \check{a}\notin c(\tau).
        \end{equation*}
    \end{claim}
    Given the claim, it follows that for any $a$ in the ground model, if there is some condition forcing $\check{a}\in c(\tau)$, then it is in fact forced by $\P$. We may define $A$ to be the set of all $a$ so that $\check{a}\in c(\tau)$ is forced, and conclude that $c(\tau) = \check{A}$ is forced by $\P$, as required.

    To prove the claim, assume towards a contradiction that there are two conditions $p,q\in\P$ so that $p\force \check{a} \in c(\tau)$ and $q\force \check{a}\notin c(\tau)$. By assumption, there is an automorphism $\gamma$ of $\P$ so that $\gamma\cdot p$ is compatible with $q$, and $\gamma$ respects $(E,\tau)$. It follows that $\gamma\cdot p \force \check{a}\in c(\gamma\cdot \tau) = c(\tau)$ and $q\force \check{a}\notin c(\tau)$, a contradiction as $\gamma\cdot p, q$ are compatible.
\end{proof}

\begin{remark}\label{remark:three generalizations}
There are three types of generalizations here, beyond what is needed for Proposition~\ref{prop: no classification E0 P(On)}, which will be crucial in the applications below:
\begin{enumerate}
    \item that $\tau$ is a virtual class;
    \item that invariants in $I$ are arbitrary subsets of the ground model (not just ordinals);
    \item that the result will be applied in specifically designed models of ZF.
\end{enumerate}
\end{remark}

\begin{proof}[Proof of Theorem~\ref{thm: no classification P(N)}]
    Assume for contradiction that $c\colon X \to I$ is a complete classification of $E$ where $I$ consists of subsets of $\mathcal{P}^\alpha(\N)$. By assumption, the invariants in $I$ are subsets of $V$ in any $\P$-generic extension. This contradicts Theorem~\ref{thm: anti classification subsets of V}.
\end{proof}
The proof of Theorem~\ref{thm: no classification P(On)} is similar.

\subsection{A slightly different presentation}

In this subsection we assume that $E$ admits an absolute complete classification (recall Definition~\ref{defn: absolute classification}). Fix an equivalence relation $E$ on $X$ and an absolute complete classification 
\begin{equation*}
    x \mapsto A_x.
\end{equation*}
\begin{defn}[Potential invariants]\label{defn: potential invariants}
Say that a set $A$ is a \textbf{potential invariant} for $E$ if in some forcing extension there is an $x$ in $X$ such that $A = A_x$. 
\end{defn}

\begin{obs}\label{obs: pot inv is virtual class}
Any potential invariant corresponds to a virtual $E$-class. Moreover, $E$-equivalence between potential invariants (thought of as virtual classes) is simply equality.
\end{obs}
\begin{proof}
Let $A$ be a potential invariant. There is a poset $\P$ and a $\P$-name $\tau$ so that $\P\force A_x = A$. It follows that $(\P,\tau)$ is a stable virtual $E$-class.

Let $A_1, A_2$ be potential invariants, and assume that the corresponding virtual $E$-classes $(\P_1,\tau_1)$, $(\P_2,\tau_2)$ are $E$-equivalent, that is, $\P_1\times\P_2$ forces that $\tau_1\mathrel{E}\tau_2$. It follows that $\P_1\times\P_2$ forces that $\check{A}_1 = A_{\tau_1} = A_{\tau_2} = \check{A}_2$, and so $A_1 = A_2$.  
\end{proof}

\begin{defn}
    Let $\P$ be a forcing poset, $\tau$ a $\P$-name. 
    \begin{enumerate}
        \item Say that $\tau$ is {\bf $\P$-ergodic} if for any $p,q\in\P$ there is an automorphism $\gamma$ of $\P$ so that $\gamma\cdot p$ is compatible with $q$ and $\gamma$ preserves $\tau$, that is, for any $\P$-generic $G$, $\tau[G] = \tau[\gamma\cdot G]$.
        \item Say that $\tau$ is {\bf non-trivial} if there are two generic filters $G,H$ so that $\tau[G] \neq \tau[H]$. Equivalently, if $\tau$ is not forced to be a member of the ground model.
    \end{enumerate}
\end{defn}
\begin{obs}
    Let $\tau$ be a name for a potential invariant. Let $\sigma$ be the name for the virtual class corresponding to $\tau$ as in Observation~\ref{obs: pot inv is virtual class}. Then
    \begin{enumerate}
        \item $\tau$ is $\P$-ergodic if and only if $\sigma$ is $(\P,E)$-ergodic.
        \item $\tau$ is non-trivial if and only if $\sigma$ is unstable.
    \end{enumerate}
\end{obs}


As a corollary we recast our main theorem from Section~\ref{Section: main result} in terms of potential invariants.

\begin{thm}[ZF]\label{thm: general P ergodicity with class invs}
    Let $\P$ be a forcing poset, $\dot{A}$ a non-trivial $\P$-name for a potential invariant for $E$. Assume that $\dot{A}$ is $\P$-ergodic. 
    \begin{enumerate}
        \item Then there is no absolute complete classification $c\colon X \to I$ of $E$ such that in any $\P$-generic extension of $V$, the invariants in $I$ are subsets of $V$.
        \item Assume further that forcing with $\P$ adds no new members of $\mathcal{P}^{\alpha}(\N)$.    
        Then there is no absolute complete classification of $E$ using subsets of $\mathcal{P}^\alpha(\N)$ as invariants. In particular, ${E}\not\leq_B {=^{+\alpha}}$.
    \end{enumerate}
\end{thm}

\section{Applications with an arbitrary ground model}

In this section we present two examples of how the main theorem may be applied to conclude irreducibility to $=^+$ (unclassifiability by countable sets of reals). In these examples, point (3) in Remark~\ref{remark:three generalizations} is not relevant. That is, Theorem~\ref{thm: anti classification subsets of V} is applied without any assumptions on the ground model $V$.


\subsection{A proof of ${=^{++}}\not\leq_B{=^+}$}
As warm-up, we present a proof of the irreducibility ${=^{++}}\not\leq_B{=^+}$ using Theorem~\ref{thm: general P ergodicity with class invs}.
\begin{remark}[Potential invariants for $=^{++}$]
    $=^{++}$ admits an absolute complete classification $x\mapsto A_x$ where the classifying invariants are all countable sets of countable sets of reals. It follows that any sets of set of reals $A$ is a potential invariant for $=^{++}$.
\end{remark}

Let $\P$ be the poset of all countable partial functions $p\colon \R\to 2$, ordered by extension. We identify a generic filter with a `generic function' $\R \to 2$, and in turn with a `generic subset' $G\subset \R$. Let $\dot{B}$ be the $\P$-name for the set $B=\set{H\subset\R}{H\Delta G\textrm{ is countable}}$, all subsets of $\R$ which have countable symmetric difference with $G$. Note that in the extension $V[G]$, $B$ is a potential invariant for $=^{++}$, as it is a set of sets of reals.
We will use the following basic facts:
\begin{itemize}
    \item Forcing with $\P$ does not add reals;
    \item The generic subset $G$ is not in the ground model.
\end{itemize}
\begin{claim}
    $\dot{B}$ is $\P$-ergodic.
\end{claim}
\begin{proof}
Given $p,q\in\P$, consider the automorphism of $\P$ acting on $r\in\P$ by flipping the value of $r(x)$ for any $x\in \dom p\cap\dom q$ for which $p(x)\neq q(x)$. Since only countably many values of $G$ are changed, the automorphism fixes $\dot{B}$.
\end{proof}
From Theorem~\ref{thm: general P ergodicity with class invs} we conclude:
\begin{cor}
    There is no absolute complete classification of $=^{++}$ using invariants which are sets of reals. In particular, ${=^{++}} \not\leq_B {=^+}$.
\end{cor}

\subsection{Subgroups of $(\R,+)$}
We present an application of Theorem~\ref{thm: no classification P(N)} to prove an irreducibility to $=^+$. In this example the poset $\P$ is a bit more intricate.

Let $\mathcal{A} \subset \R^\N$ be all injective sequences $\seqq{x_n}{n\in\N}$ so that $x_0 = 0$ and $(\set{x_n}{n\in\N},+)$ is a subgroup of $(\R,+)$. $\mathcal{A}$ is a closed subset of $\R^\N$, equipped with the product topology.
\begin{defn}
    Define the equivalence relation $\sim_\mathcal{A}$ on $\mathcal{A}$ as follows. 
    \begin{center}
       $\seqq{x_n}{n\in\N} \sim_\mathcal{A} \seqq{y_n}{n\in\N} \iff$ \\ $\exists \lambda > 0$ so that the map $x\mapsto \lambda\cdot x$ is a bijection from $\set{x_n}{n\in\N}$ to $\set{y_n}{n\in\N}$.      
    \end{center}

\end{defn}
This equivalence relation was studied in~\cite{Anti-classification-CMMS}. The main motivation is that $\sim_{\mathcal{A}}$ is equivalent to the classification problem for countable Archimedean groups. (See Proposition~3.1 in \cite{Anti-classification-CMMS} and the discussion following it.) In particular, it is established in~\cite[Section~4]{Anti-classification-CMMS} that
\begin{equation*}
 {=^+} <_B {\sim_{\mathcal{A}}} <_B {=^{++}}
\end{equation*}
We provide here a proof of the lower bound.\footnote{The proof here is shorter and slightly simpler, though the proof in~\cite{Anti-classification-CMMS} gave a stronger lower bound, see~\cite[Theorem~1.2]{Anti-classification-CMMS}.}
\begin{prop}\label{prop: ArchGrps not red to =+}
    $\sim_\mathcal{A}$ is not Borel reducible to $=^+$.
\end{prop}
This lower bound was recently used by \cite{Gaoetal-extremely-amenable-2024}, where they study the complexity of classification for extremely amenable groups.

We introduce some terminology and observations from \cite[4.3]{Anti-classification-CMMS}.
We identify a member $\seqq{x_n}{n\in\N}\in \mathcal{A}$ with the subgroup $G=\set{x_n}{n\in\N}$.
\begin{remark}\label{remark: observation hion lemma}
Suppose $G_1, G_2$ are subgroup of $(\R,+)$ and $x\mapsto \lambda\cdot x$ is a bijection from $G_1$ to $G_2$. Then $\lambda = b / a$ for some $b\in G_2$ and $a\in G_1$. 
\end{remark}

\begin{remark}[Virtual $\sim_{\mathcal{A}}$-classes]

Any subgroup $G$ of $(\R,+)$ (not necessarily countable) is naturally identified as a virtual $\sim_\mathcal{A}$-class. Specifically, the corresponding virtual class is $(\Q,\sigma)$ where $\Q$ collapses $|G|$ to be countable and $\sigma$ is a name for an injective enumeration of $G$. For two subgroups $G_1, G_2$, thought of as virtual classes, they are $\sim_{\mathcal{A}}$-equivalent if and only if there is some $\lambda$ so that $x\mapsto\lambda x$ is a bijection from $G_1$ to $G_2$. Note that the definition (equivalence of virtual classes) only gives us this $\lambda$ in some forcing extension. Nevertheless, by the remark above, $\lambda = b/a$ for some $b\in G_2, a\in G_1$, and so $\lambda$ exists in any model in which $G_1,G_2$ exists.
    
\end{remark}

Now we define the forcing $\P$ to introduce a new virtual $\sim_\mathcal{A}$-class.
Given a set $S\subset \R$, define $D(S)$ to be the set of all numbers of the form $a_1^{l_1}\cdot,\dots,\cdot a_n^{l_n}$, where $l_1,\dots,l_n$ are integers and $a_1,\dots,a_n$ are in $S$. 
Let $Y$ be a perfect set of reals so that any $a_1,\dots,a_n \in Y$ are algebraic independent over $Y\setminus\{a_1,\dots,a_n\}$. In particular, for any two subsets $S_1,S_2\subset Y$, 
\begin{equation*}
    D(S_1) \cap D(S_2) = D(S_1 \cap S_2).
\end{equation*}

\begin{defn}
    Let $\P$ be the poset of conditions $p = (S_p,G_p)$ where $S_p$ is a countable subset of $Y$ and $G_p$ is a countable subset of $D(S_p)$. For $p,q\in\P$ say that $q$ extends $p$ if $S_q\supseteq S_q$ and $G_q \cap D(S_p) = G_p$.
\end{defn}
$\P$ is a partial order with maximal element $(\emptyset,\emptyset)$.
Given a generic filter $F\subset \P$, let $G$ be the additive subgroup of $\R$ generated by $\bigcup_{p\in F}G_p$. 

\begin{claim}
    $\dot{G}$ is an unstable $(\P,\sim_{\mathcal{A}})$-ergodic name for a virtual $\sim_{\mathcal{A}}$-class.
\end{claim}
\begin{proof}
First we show $(\P,\sim_{\mathcal{A}})$-ergodicity.
Fix $a\in Y$. Let $\P_a = \set{p\in\P}{a\in S_p}$, a dense open subset of $\P$.
Define $\pi_a(p) = (S_p,a\cdot G_p)$. Then $\pi_a$ is an automorphism of $\P_a$ which respects $(\sim_{\mathcal{A}},\dot{G})$.

Given any $p,q\in\P$, fix some $a\in Y\setminus (S_p\cup S_q)$. By the algebraic independence of the members of $Y$, it follows that $a \cdot D(S_p) \cap D(S_q) = \emptyset$, and therefore $\pi_a(p)$ and $q$ are compatible.

Next, we show that $\dot{G}$ is unstable. Suppose $F_1,F_2$ are mutually $\P$-generics. We need to show that $\dot{G}[F_1]\not\sim_{\mathcal{A}}\dot{G}[F_2]$. By Remark~\ref{remark: observation hion lemma}, it suffices to show that for any $a\in \dot{G}[F_1]$ and $b\in \dot{G}[F_2]$, the map $x\mapsto b/a \cdot x$ is not a bijection from  $\dot{G}[F_1]$ and $\dot{G}[F_2]$.

It suffices to show the following. Given $p,q\in \P$, $a,b\in D(Y)$ such that $a\in G_p$ and $b\in G_q$, we need to find extensions $p', q'$ of $p,q$ respectively, some $a' \in G_{p'}$ so that $a',a,b\in S_{q'}$ yet $a'\cdot (b/a)\notin G_{q'}$. (In this case, $(p',q')$ forces that $x\mapsto b/a \cdot x$ is not a bijection from $\dot{G}[F_1]$ and $\dot{G}[F_2]$.)

Since $S_p,S_q$ are countable, there is some $x\in Y\setminus (S_p\cup S_q)$. By the algebraic independence of the members of $Y$ it follows that $x\notin D(S_p\cup S_q)$. Define
\begin{equation*}
    S_{p'} = S_{q'} = S_p\cup S_q \cup \{x\}, \quad G_{p'} = G_p \cup \{x\}, \quad G_{q'} = G_q.
\end{equation*}
Then $p', q'$ are as desired.
\end{proof}
Since $\P$ is countably closed, no reals are added when forcing with $\P$. It follows from Theorem~\ref{thm: no classification P(N)} that $\sim_{\mathcal{A}}$ is not Borel reducible to $=^+$, concluding the proof of Proposition~\ref{prop: ArchGrps not red to =+}.

\section{Applications over choiceless ground models}

In this section we present applications of Theorem~\ref{thm: anti classification subsets of V} where the poset $\P$ and name $\tau$ for a virtual $E$-class are constructed in a specific choiceless model. In these examples we naturally end up with the stronger anti-classification Theorem~\ref{thm: no classification P(On)}.

\subsection{Forcing over the basic Cohen model}\label{subsec: forcing over basic Cohen model}

Recall `the basic Cohen model' in which the axiom of choice fails. The model is of the form $V(A)$ where $A$ is a set of reals. See Appendix~\ref{Appendix: basic cohen model}.
One basic fact is that in $V(A)$, $A$ is a Dedekind-finite set, that is, there is no sequence of distinct members of $A$. In particular, there is no $y\in \R^\N$ enumerating $A$. It follows that $A$ is a non-trivial potential invariant for $=^+$, in $V(A)$.

\begin{defn}[The poset $\P$]
Working in $V(A)$, let $\P$ be the poset of all finite partial functions $p\colon A\to 2$, ordered by extension. 
\end{defn}
We identify a generic filter with a `generic function' $A \to 2$, and in turn with a `generic subset' $G\subset A$. 
The key point here is the following, due to Monro~\cite{Mon73}, which we prove in Appendix~\ref{Appendix: basic cohen model} below.
\begin{lemma}\label{lemma: no new reals adding one subset}
Forcing with $\P$ over $V(A)$ adds no subsets of $V$.
\end{lemma}
In particular, no sets of ordinals are added. It follows that any set of sets of ordinals, in a $\P$-generic extension of $V(A)$, is a subset of $V(A)$.

\subsubsection{Another proof of ${=^{++}\not\leq_B{=^+}}$}

Let $\dot{B}$ be a $\P$-name for the set 
\begin{equation*}
B = \set{H\subset A}{H\triangle G \textrm{ is finite}},    
\end{equation*}
all subsets of $A$ which have finite symmetric difference with $G$. Note that, in the extension $V(A)[G]$, $B$ is a potential invariant for $=^{++}$, as it is a set of sets of reals.

\begin{claim}[In $V(A)$]\label{claim: ergodic inv between + and ++}
    $\dot{B}$ is $\P$-ergodic.
\end{claim}
\begin{proof}
    Given a finite subset $F\subset A$, there is an automorphism $\pi_F$ of $\P$, flipping the values on $F$. That is, for $a\in F$, $p\in \P$, and $a\in \dom p$, $\pi_F(p)(a) = 1 - p(a)$. Note that $\pi_F$ preserves $\dot{B}$, as it changes only finitely many values of $\dot{G}$. Finally, for any two conditions $p,q\in\P$ there is a finite set $F$ so that $\pi_F(p)$ is compatible with $q$.
\end{proof}
We conclude from Theorem~\ref{thm: general P ergodicity with class invs} that, in $V(A)$, there is no absolute complete classification of $=^{++}$ using invariants which are subsets of $V(A)$ in any $\P$-generic extension.

\begin{cor}\label{cor: =++ unclass sets of sets ord}
    There is no absolute complete classification of $=^{++}$ using invariants which are sets of sets of ordinals. In particular, ${=^{++}} \not\leq_B {=^+}$.
\end{cor}
 
\subsubsection{An equivalence relation between $=^+$ and $=^{++}$}

We present a more subtle irreducibility result, using the same model $V(A)$ and poset $\P$.
\begin{defn}
Consider the space $X = \R^\N \times 2^\N$. For $(x,t)\in X$, define $A_x = \set{x(n)}{n\in\N}$ and $G_{(x,t)}=\set{x(n)}{t(n)=1}$. Define
\begin{equation*}
    (x,t) \mathrel{E} (y,s) \iff A_{x} = A_{y} \textrm{ and } G_{(x,t)} \Delta G_{(y,s)} \textrm{ is finite}.
\end{equation*}
\end{defn}
It is readily seen that ${=^+}\leq_B{E}\leq_B{=^{++}}$. $E$ is an example of an equivalence relation with \emph{potential complexity} $D(\mathbf{\Pi}^0_3)$.\footnote{See \cite{HKL98} for a treatment of potential complexity. The equivalence relation $E$ here is Borel reducible to the equivalence relation $\cong^\ast_{3,1}$ there. See also \cite{Shani-reducibility-symmetric-models-21, Shani-FS-homomorphisms}.} It follows from the results in~\cite{HKL98} that ${E} <_B {=^{++}}$. We show here that the reduction $=^+ \leq_B E$ is strict as well.

\begin{prop}\label{claim: finite diff subset of A anticlassification}
   There is no absolute complete classification of $E$ using invariants which are sets of sets of ordinals. In particular, ${E} \not\leq_B {=^+}$. 
\end{prop}

\begin{remark}[Virtual $E$-classes]
    Any pair $(A,G)$ where $A\subset \R$ and $G\subset A$ may be viewed as a virtual $E$-class. Two virtual $E$-classes $(A_1,G_1)$, $(A_2,G_2)$ are $E$-equivalent if and only if $A_1 = A_2$ and $G_1\Delta G_2$ is finite.
\end{remark}

\begin{claim}
    Working in $V(A)$,    $(A,\dot{G})$ is an unstable $(\P,E)$-ergodic name for a virtual $E$-class.
\end{claim}
\begin{proof}
    The proof of Claim~\ref{claim: ergodic inv between + and ++} shows that $(A,\dot{G})$ is $(\P,E)$-ergodic, since each $\pi_F$ respects $(E,(A,\dot{G}))$. For mutually $\P$-generics $G,H$, $G\Delta H$ is infinite, and therefore $(A,G), (A,H)$ are not equivalent virtual $E$-classes. That is, $(A,\dot{G})$ is unstable.
\end{proof}

Since $\P$ adds no sets of ordinals over $V(A)$, Proposition~\ref{claim: finite diff subset of A anticlassification} now follows from Theorem~\ref{thm: no classification P(On)}.

\subsection{Forcing over the second Cohen model}\label{subsec: second cohen model}

Recall the `second Cohen model' in which choice fails for a sequence of pairs. See Appendix~\ref{Appendix: second cohen model}. The model is of the form $V(\seqq{Z_n}{n\in\N})$ where each $Z_n$ is a set two $E_0$-classes. That is, $Z_n = \{[a]_{E_0},[b]_{E_0}\}$ for some $a,b\in 2^\N$. In this model $\prod_n Z_n = \emptyset$. We present an application of our main theorem where $V(\seqq{Z_n}{n\in\N})$ is used as the base model.

\begin{defn}\label{defn: E pairs finite swaps}
Let $X=((2^\N)^2)^\N$. For $x\in X$ write $x=\seqq{x(n)}{n\in\N}$, where $x(n)=(x(n)(0),x(n)(1))\in(2^\N)^2$. Define the relation $E$ on $X$ by $x\mathrel{E}y$ if and only if 
\begin{equation*}
    \{[x(n)(0)]_{E_0}, [x(n)(1)]_{E_0}\} =     \{[y(n)(0)]_{E_0}, [y(n)(1)]_{E_0}\}\textrm{, for all }n\textrm{ and}
\end{equation*}
\begin{equation*}
    x(n)(0) \mathrel{E_0} y(n)(0)\textrm{, for all but finitely many }n.
\end{equation*}
\end{defn}

In other words, $x\mathrel{E}y$ if $x$ and $y$ are pointwise $E_0$-equivalent, up to finitely many swaps of the two coordinates. Note that $E$ is a Borel equivalence relation.

\begin{prop}\label{prop: unclass E over second cohen model}
       There is no absolute complete classification of $E$ using invariants which are sets of sets of ordinals. In particular, ${E} \not\leq_B {=^+}$. 
\end{prop}

\begin{remark}[Virtual $E$-classes]\label{remark: virtual classes finite swaps}
Any sequence $\seqq{(A_n,B_n)}{n\in\N}$ of pairs of $E_0$-classes can be viewed as a virtual $E$-class. Two such virtual $E$-classes $\seqq{(A_n,B_n)}{n\in\N}$, $\seqq{(C_n,D_n)}{n\in\N}$ are $E$-equivalent if they are equal up to finitely many swaps.
\end{remark}

\begin{defn}[The poset $\P$]
Working in $V(\seqq{Z_n}{n\in\N})$, let $\P$ be the poset of finite partial choice functions in $\prod_n Z_n$, ordered by extension.    
\end{defn}

The key point here is the following, which we prove in Appendix~\ref{Appendix: second cohen model} below.
\begin{lemma}\label{lemma: no new sets second cohen}
Forcing with $\P$ over $V(\seqq{Z_n}{n\in\N})$ adds no subsets of $V$. In particular, no sets of ordinals are added.
\end{lemma}

Let $\dot{z}$ be the $\P$-name for the generic choice function added by $\P$. Let $\tau$ be the $\P$-name for a sequence of pairs of $E_0$-classes defined by $\tau(n)(0) = z(n)$ and $\tau(n)(1)$ is the other member of $Z_n$. Then $\tau$ is a $\P$-name for a virtual $E$-class.

\begin{claim}\label{claim: ergodic name second cohen}
In $V(\seqq{Z_n}{n\in\N})$, $\tau$ is an unstable $(\P,E)$-ergodic name for a virtual $E$-class.
\end{claim}
\begin{proof}
For a finite set $F\subseteq\N$, let $\pi_F$ be the automorphism of $\P$ which replaces $p(n)$ by the other member of $Z_n$, for $n\in F\cap\dom p$, and fixes $p$ otherwise. Note that $\pi_F$ respects $(E,\tau)$.
Given $p,q\in\P$, $\pi_F(p)$ and $q$ are compatible for $F=\set{n\in\dom p\cap\dom q}{p(n)\neq q(n)}$. This concludes that $\tau$ is $(\P,E)$-ergodic.

Finally, a density argument shows that $\P\times\P$ forces that the set of $n\in\N$ on which $\tau_l$ and $\tau_r$ disagree is infinite, and therefore $\P\times\P\force\tau_l\not\mathrel{E}\tau_r$.
\end{proof}
Proposition~\ref{prop: unclass E over second cohen model} now follows from Theorem~\ref{thm: no classification P(On)}. 

\subsection{Another example: $\Z$-shift on $(E_0^\N)^\Z$}\label{subsec: Z shift}


\begin{defn}\label{defn: E Z jump}
Let $X=((2^\N)^\N)^\Z$. For $x\in X$ and $k\in\Z$ let $k\cdot x$ be the shift of $x$ by $k$, defined by $(k\cdot x)(n)=x(n-k)$. Define the relation $E$ on $X$ by
\begin{equation*}
x\mathrel{E}y \iff \exists k\in\Z\ \forall n,m\in\N\ \big( x(n-k)(m)\mathrel{E_0} y(n)(m)\big),
\end{equation*}
that is, $x\mathrel{E}y$ if and only if some shift of $x$ is pointwise $E_0$-equivalent to $y$. Note that $E$ is a Borel equivalence relation.
\end{defn}

\begin{prop}\label{prop: unclass E Z jump}
       There is no absolute complete classification of $E$ using invariants which are sets of sets of ordinals. In particular, ${E} \not\leq_B {=^+}$.
\end{prop}

\begin{remark}
    The equivalence relation $E$ is $(E_0^\N)^{[\Z]}$, the \emph{$\Z$-jump of $E_0^\N$}, as defined by Clemens and Coskey~\cite{ClemensCoskey2022}. The irreducibility ${E}\not\leq_B {=^+}$ was proven in~\cite[Section~6]{Allison-Shani-abelian-products}, answering a question from~\cite{ClemensCoskey2022}.
\end{remark}

\begin{remark}[Virtual $E$-classes]\label{remark: virtual classes Z jump}
Any $\Z\times\N$-array of $E_0$-classes can be viewed as a virtual $E$-class. Two such virtual $E$-classes are $E$-equivalent if they are equal up to a $\Z$-shift of the array.
\end{remark}

We will construct a model with the following properties. The model is of the form 
\begin{equation*}
V(\seqq{A_p}{p\textrm{ prime}})    
\end{equation*}
where each $A_p$ is the set of all cyclic permutation of a sequence of $p$ many $E_0$-classes. That is, there is some sequence $\bar{a} = a_0, \dots, a_{p-1}$, where each $a_i$ is an $E_0$-class, so that $A_p = \set{k\cdot \bar{a}}{k < p}$, where $k\cdot \bar{a}$ is the cyclic permutation of $\bar{a}$ by $k$. In this model, $\prod_p A_p = \emptyset$. See Appendix~\ref{Appendix: cyclic permutations model} for the construction.

\begin{defn}[The poset $\P$]
Working in $V(\seqq{A_p}{p\textrm{ prime}})$, let $\P$ be the poset of finite partial choice functions in $\prod_p A_p$, ordered by extension.
\end{defn}

The key point is the following lemma (see~Appendix~\ref{Appendix: cyclic permutations model}).
\begin{lemma}\label{lemma: no new sets CRT model}
Forcing with $\P$ over $V(\seqq{A_p}{p\textrm{ prime}})$ adds no subsets of $V$. In particular, no sets of ordinals are added.
\end{lemma}

Let $\dot{y}$ be the $\P$-name for the generic choice function added by $\P$. 
Let $\tau$ be the $\P$-name for a $\Z\times\N$ array of $E_0$-classes defined so that
\begin{equation*}
\tau(k)(n)=\dot{y}(p)(k \bmod p),
\end{equation*}
where $p$ is the $n$'th prime. That is, the $n$'th row of $\tau$ is the $p$-tuple $\dot{y}(p)$ repeated periodically along the $\Z$-axis. Then $\tau$ is a $\P$-name for a virtual $E$-class.

\begin{claim}\label{claim: ergodic name Z jump}
In $V(\seqq{A_p}{p\textrm{ prime}})$, $\tau$ is an unstable $(\P,E)$-ergodic name for a virtual $E$-class.
\end{claim}
\begin{proof}
For $k\in\Z$, let $\gamma_k$ be the automorphism of $\P$ which replaces each chosen member of $A_r$ by its cyclic rotation by $k$, for every prime $r$. Note that $\gamma_k$ shifts each coordinate of the array $\tau$ by $k$, and therefore respects $(E,\tau)$.
Given $p,q\in\P$, for each prime $r\in\dom p\cap\dom q$ there is some $k_r\in\Z$ taking $p(r)$ to $q(r)$. By the Chinese Remainder Theorem there is some $k\in\Z$ with $k\equiv k_r \pmod{r}$ for all such $r$. Now $\gamma_k(p)$ and $q$ agrees on their common domain, and are therefore compatible. This concludes that $\tau$ is $(\P,E)$-ergodic.

Finally, given any $(p,q)\in\P\times\P$ and $k\in\Z$, we may choose a prime $r>|k|$ outside the domains of $p,q$, and extend $p,q$ to $p',q'$ so that $(p',q')$ forces that the $r$'th rows of $\tau_l$ and $k\cdot\tau_r$ are not equal. A density argument therefore shows that $\P\times\P\force\tau_l\not\mathrel{E}\tau_r$.
\end{proof}
Proposition~\ref{prop: unclass E Z jump} now follows from Theorem~\ref{thm: no classification P(On)}.

\section{Appendix}

\subsection{Abstract nonsense}

Suppose $A$ is a set in some generic extension of $V$. Let $V(A)$ be the minimal transitive model of ZF extending $V$ and containing the set $A$. For example, this model can be written as the class directed union of $L(A,x)$ for $x\in V$.
We think of this as the model \textbf{generated by $A$} (over $V$).

There are similar ways of ``generating models'' from a set $A$: $L(A)$, $\hod_{\mathrm{tc}\{A\}}$, $\hod_{V,\mathrm{tc}\{A\}}$, as well as various symmetric models. See \cite[Chapter~13]{Jec03} for ordinal definability and \cite[Chapter~15]{Jec03}, \cite[Chapter~17]{Halbeisen2017-CombinatorialSetTheory} for symmetric models. 
For the $A$'s we study here, all these models will have the same relevant properties. Also, in our examples the models $V(A)$ and $\hod_{V,\mathrm{tc}\{A\}}$ will coincide.

The minimality of $V(A)$ implies the following `definable closure' property.
\begin{fact}
\label{fact;definability-V(A)}
The following holds in $V(A)$.
For any set $X$, there is some formula $\psi$, parameters $\bar{a}$ from the transitive closure of $A$ and $v\in V$ such that $X$ is the unique set satisfying $\psi(X,A,\bar{a},v)$.
Equivalently, there is a formula $\varphi$ such that $X=\set{x}{\varphi(x,A,\bar{a},v)}$.
\end{fact}
\begin{proof}
    Work in the model $V(A)$. Consider the model $\mathrm{HOD}_{V,\mathrm{tc}\{A\}}$ of all hereditarily definable sets using parameters which are members of $V$, the transitive closure of $A$, as well as the set $A$ itself. (That is, all sets definable as in the statement of the fact.) Then $\mathrm{HOD}_{V,\mathrm{tc}\{A\}}$ is a transitive model of ZF (see~\cite[Theorem~13.26]{Jec03}), which extends $V$ and contains $A$. By minimality of $V(A)$, we conclude that it is equal to $\mathrm{HOD}_{V,\mathrm{tc}\{A\}}$, and so every set is definable in this way.
\end{proof}

\subsection{The basic Cohen model}\label{Appendix: basic cohen model}
We recall here `the basic Cohen model' in which the axiom of choice fails (see \cite{Kan08}). 

Let $\C$ be Cohen forcing for adding a member to $\R^\N$. A condition in $\C$ is a finite sequence $p$ so that $p(i)$ is a non-empty open subset of $\R$, for each $i$ in its domain. Say that a condition $p$ extends $q$ if the domain of $p$ extends the domain of $q$ and $p(i)\subset q(i)$ for each $i\in \dom q$.\footnote{$\C$ is (equivalent to) the topological version of Cohen forcing for the space $\R^\N$. See~\cite[Section~26]{Jec03}.}

A generic filter $G\subset \C$ is naturally identified with a `generic real' $x\in\R^\N$, where $\{x(i)\} = \bigcap_{p\in G}p(i)$. 
Let $x\in\R^\N$ be a $\C$-`generic real' over $V$. Consider the model
\begin{equation*}
V(A) , \textrm{where } A = \set{x(n)}{n\in\N}.    
\end{equation*}
 
We will use the following key lemma. See \cite[p.133]{Felgner-ZF-set-thery-1971}, also \cite[p.19]{Anti-classification-CMMS} for a proof in this presentation of the basic Cohen model.

\begin{lemma}[Continuity Lemma]\label{lem;continuity-Cohen-model}
Let $\phi$ be a formula, $\bar{a}=a_0,...,a_{n-1}$ a finite sequence of distinct members of $A$, and $v\in V$. Suppose $\phi(A,\bar{a},v)$ holds in $V(A)$.
Then there are open sets $U_0,...,U_{n-1}$ such that $a_i\in U_i$ and for any $\bar{b} = b_0,...,b_{n-1}$ consisting of distinct elements from $A$, if $b_i\in U_i$ for all \( i < n \), then $\phi(A,\bar{b},v)$ holds in $V(A)$.
\end{lemma}

\begin{fact}[V(A)]
    The set $A$ is Dedekind-finite set, that is, there is no sequence of distinct members of $A$.
\end{fact}

The key technical point, used in the applications above, was Lemma~\ref{lemma: no new reals adding one subset}. Note that the Lemma itself witnesses failure of choice in $V(A)$, and in particular implies that $A$ is Dedekind finite. We provide a proof here.

\begin{proof}[Proof of Lemma~\ref{lemma: no new reals adding one subset}]
Suppose $\tau$ is a $\P$-name for a subset of $V$, $\tau\in V(A)$. As in Fact~\ref{fact;definability-V(A)}, let $\bar{a}\subset A$ be finite, $w\in V$ a parameter, and $\psi$ a formula so that in $V(A)$, $\tau$ is defined as the unique solution to $\psi(\tau,A,\bar{a},w)$.

We will show that for any condition $p\in\P$ and any $v\in V$, if $p$ forces $\check{v}\in\tau$ then $p\restriction \bar{a}$ forces the same. It then follows that for any generic filter $G\subset P$ over $V(A)$, the values of $\tau[G]$ are determined by the single condition $q = G\restriction \bar{a}$. That is, $\tau[G]$ may be defined in $V(A)$ as the set of all $v$ for which $q\force \check{v}\in\tau$.

Fix a condition $p\in \P$ and $v\in V$ so that $p\force \check{v}\in\tau$. Assume that the domain of $p$ is of the form $\bar{a},\bar{b}$, with $\bar{b}$ disjoint from $\bar{a}$.
By the Continuity Lemma, there are infinitely many distinct tuples $\bar{b}'$, disjoint from $\bar{a}$, such that $p[\bar{b}']\force v\in\tau$ as well, where $p[\bar{b}']$ is the condition with domain $\bar{a},\bar{b}'$ defined on $\bar{b}'$ as $p$ is defined on $\bar{b}$.
Note that $p[\bar{b}']$ extends $p\restriction \bar{a}$.

Now for any $r$ extending $p\restriction \bar{a}$, there is some $\bar{b}'$ such that $p[\bar{b}']\force \check{v}\in\tau$ and $p[\bar{b}']$ is compatible with $r$ (take $\bar{b}'$ disjoint from the domain of $r$).
It follows that $p\restriction \bar{a}$ forces $\check{v}\in\tau$.
\end{proof}

\subsection{The second Cohen model}\label{Appendix: second cohen model}

We recall here `the second Cohen model', in which the axiom of choice fails for a countable sequence of pairs (see~\cite{Kan08}).

Let $\C$ be Cohen forcing for adding a member of the space $((2^\N)^2)^\N$. A condition in $\C$ is a finite sequence $p$ so that for each $n\in\dom p$, $p(n)=(p(n)(0),p(n)(1))$ is a pair of finite binary sequences. 
Equivalently, we may think of $p$ as a finite partial function from $\N\times 2 \times \N$ to $2$. Say that a condition $p$ extends $q$ if $p$ extends $q$ as a partial function.\footnote{$\C$ is (equivalent to) the topological version of Cohen forcing for the space $((2^\N)^2)^\N$.}

A generic filter $G\subset \C$ is naturally identified with a `generic sequence' $x=\seqq{x(n)}{n\in\N}\in((2^\N)^2)^\N$. 
Let $x\in((2^\N)^2)^\N$ be $\C$-generic over $V$. For $n\in\N$ set
\begin{equation*}
Z_n=\{[x(n)(0)]_{E_0},[x(n)(1)]_{E_0}\}.
\end{equation*}
Let $\mathcal{Z}=\seqq{Z_n}{n\in\N}$. Consider the model
\begin{equation*}
V(\mathcal{Z}) = V(\seqq{Z_n}{n\in\N})\subseteq V[x].
\end{equation*}
This is, essentially, `the second Cohen model'.
\begin{fact}[$V(\seqq{Z_n}{n\in\N})$]\label{fact: no choice in second Cohen}
$\prod_{n\in\N}Z_n=\emptyset$.
\end{fact}



The key point is that the two members of each pair are sufficiently indiscernible.

\begin{lemma}[Indiscernibility]\label{lem: indiscernibility second Cohen}
Let $\phi$ be a formula, $v\in V$, and $\bar{a}=a_0,\ldots,a_{k-1}$ a finite sequence of members of $\bigcup_{n\in\N}Z_n$ including at most one member of each $Z_n$. Suppose that $\phi(\mathcal Z,\bar{a},v)$ holds in $V(\mathcal Z)$.
Then for any $b_0,\ldots,b_{k-1}\in\bigcup_{n\in\N}Z_n$ so that $b_i\in Z_n$ if and only if $a_i\in Z_n$, the formula $\phi(\mathcal Z,\bar{b},v)$ holds in $V(\mathcal Z)$ as well.
\end{lemma}
\begin{proof}
First note that if $Z_n=\{a,b\}$ then $b$ is definable from $Z_n$ and $a$. Thus we may assume that $a_i=[x(n_i)(0)]_{E_0}$ for some $n_0,\ldots,n_{k-1}\in\N$.

By assumption, $\phi^{V(\mathcal Z)}(\mathcal Z,\bar{a},v)$ holds in $V[x]$. Fix a condition $p\in\C$ forcing that
\begin{equation*}
\text{$\phi^{V(\dot{\mathcal Z})}(\dot{\mathcal Z},[\dot{x}(n_0)(0)]_{E_0},\ldots,[\dot{x}(n_{k-1})(0)]_{E_0},\check{v})$}.
\end{equation*}
We can find a $\C$-generic sequence $y$ over $V$ such that:
\begin{enumerate}
\item $[y(n_i)(0)]_{E_0}=b_i$ for $i<k$; and
\item $y$ extends the condition $p$.
\end{enumerate}
The first requirement is achieved by swapping $x(n_i)(0), x(n_i)(1)$ for $i<k$ with $a_i\neq b_i$. The second requirement is then achieved by making finitely many changes to finitely many $y(n)(j)\in 2^\N$.
Note that this second step does not change $[y(n)(j)]_{E_0}$ for any $n,j$ and in particular does not change $\mathcal{Z}$. The first step also fixes $\mathcal{Z}$.

Finally, since $y$ extends $p$, working in $V[y]$ we have that
\begin{equation*}
\phi^{V(\mathcal Z)}(\mathcal Z,[y(n_0)(0)]_{E_0},\ldots,[y(n_{k-1})(0)]_{E_0},v),
\end{equation*}
that is, $\phi^{V(\mathcal Z)}(\mathcal Z,b_0,\ldots,b_{k-1},v)$, as required.
\end{proof}

The key technical point, used in \S\ref{subsec: second cohen model}, was Lemma~\ref{lemma: no new sets second cohen}. Note that the Lemma itself witnesses failure of choice and implies Fact~\ref{fact: no choice in second Cohen}. We provide a proof here.

\begin{proof}[Proof of Lemma~\ref{lemma: no new sets second cohen}]
Suppose $\tau$ is a $\P$-name for a subset of $V$, $\tau\in V(\mathcal Z)$. As in Fact~\ref{fact;definability-V(A)}, let $\bar{a}$ be a finite sequence of members of $\bigcup_n Z_n$, $w\in V$ a parameter, and $\psi$ a formula, so that in $V(\mathcal Z)$, $\tau$ is defined as the  unique solution to $\psi(\tau,\mathcal Z,\bar{a},w)$. We may assume that $\bar{a}$ includes at most one member out of each $Z_n$. Let $a\subseteq\N$ be a finite set so that $\bar{a}$ consists of members of $\bigcup_{n\in a}Z_n$.

We will show that for any condition $p\in\P$ and any $v\in V$, if $p$ forces $\check{v}\in\tau$ then $p\rest a$ forces the same. It then follows that for any generic filter $G\subset\P$ over $V(\mathcal Z)$, the values of $\tau[G]$ are determined by the single condition $q=G\rest a$. That is, $\tau[G]$ may be defined in $V(\mathcal Z)$ as the set of all $v$ for which $q\force\check{v}\in\tau$.

Fix a condition $p\in\P$ and $v\in V$ so that $p\force\check{v}\in\tau$, and let $r$ be any condition extending $p\rest a$. We need to show that $r\force \check{v}\in\tau$. We may assume that $\dom r$ extends $\dom p$. Let $q$ be the restriction of $r$ to the domain of $p$. It suffices to prove that $q\force \check{v}\in\tau$.
Let $p^*, q^*$ be the restrictions of $p, q$ to $\dom p\setminus a$.

Note that the statement `$p\force \check{v}\in \tau$' can be written in $V(\mathcal{Z})$ as $\phi(\mathcal{Z},\bar{a},p^*,v,w)$ for some formula $\phi$. 
Applying the indicernibility Lemma~\ref{lem: indiscernibility second Cohen}, we conclude that $\phi(\mathcal{Z},\bar{a},q^*,v,w)$, implying that $q\force \check{v}\in \tau$, as required.
\end{proof}

\subsection{Another failure of choice for a sequence of finite sets}\label{Appendix: cyclic permutations model}

We construct here a model in which the axiom of choice fails for a countable sequence of finite sets, yet for each $N\in\N$, `choice for sequences of sets of size $\leq N$' holds.

Let $\C$ be the poset for adding a member of the space $\prod_{p\textrm{ prime}}(2^\N)^p$ by finite approximations. A condition in $\C$ is a finite function $p$ whose domain is a finite set of primes, so that for each prime $r\in\dom p$, $p(r)=(p(r)(0),\ldots,p(r)(r-1))$ is an $r$-tuple of finite binary sequences. Equivalently, we may think of $p$ as a finite partial function from $\set{(r,j,i)\in\N^3}{r\textrm{ prime},\,j<r}$ to $2$. Say that a condition $p$ extends $q$ if $p$ extends $q$ as a partial function.\footnote{$\C$ is (equivalent to) the topological version of Cohen forcing for the space $\prod_{p\textrm{ prime}}(2^\N)^p$.}

A generic filter $G\subset \C$ is naturally identified with a `generic sequence' $x=\seqq{x_p}{p\textrm{ prime}}$, with $x_p=(x_p(0),\ldots,x_p(p-1))\in(2^\N)^p$.
Let $x$ be $\C$-generic over $V$. For each prime $p$ set
\begin{equation*}
A_p=\set{\seqq{[x_p(k+i\mod p)]_{E_0}}{i<p}}{k<p}.
\end{equation*}
That is, $A_p$ is the set of cyclic permutations of the $p$-tuple $\seq{[x_p(0)]_{E_0},\ldots,[x_p(p-1)]_{E_0}}$.
Let $\mathcal A=\seqq{A_p}{p\textrm{ prime}}$. Consider the model
\begin{equation*}
V(\mathcal A) = V(\seqq{A_p}{p\textrm{ prime}})\subseteq V[x].
\end{equation*}
\begin{fact}[$V(\seqq{A_p}{p\textrm{ prime}})$]\label{fact: no choice in CRT model}
$\prod_{p}A_p=\emptyset$.
\end{fact}
\begin{remark}
    It can be shown that in $V(\mathcal{A})$, for any $N\in\N$, if $\seqq{B_n}{n\in\N}$ is a sequence so that $|B_n|\leq N$ for all $n$, then $\prod_n B_n\neq\emptyset$.
\end{remark}

As before, we will rely on the fact that the members of each $A_p$ are sufficiently indiscernible.

\begin{lemma}[Indiscernibility]\label{lem: indiscernibility CRT model}
Let $\phi$ be a formula, $v\in V$, and $\bar{a}=a_0,\ldots,a_{m-1}$ a finite sequence of members of $\bigcup_{p\textrm{ prime}}A_p$, including at most one member of each $A_p$. Suppose that $\phi(\mathcal A,\bar{a},v)$ holds in $V(\mathcal A)$.
Then for any $b_0,\ldots,b_{m-1}\in\bigcup_{p\textrm{ prime}}A_p$ so that $b_i\in A_p$ if and only if $a_i\in A_p$, the formula $\phi(\mathcal A,\bar{b},v)$ holds in $V(\mathcal A)$ as well.
\end{lemma}
We omit the proof here as it is quite similar to the arguments in Section~\ref{Appendix: second cohen model}.
The key technical point used in \S\ref{subsec: Z shift} was Lemma~\ref{lemma: no new sets CRT model}. Note that the lemma itself witnesses failure of choice and implies Fact~\ref{fact: no choice in CRT model}. We omit the proof as it is quite similar to the proof of Lemma~\ref{lemma: no new sets second cohen} in Section~\ref{Appendix: second cohen model}.

\bibliographystyle{alpha}
\bibliography{bibliography}

\end{document}